\documentclass[reqno]{amsart}
\usepackage{amscd}
\usepackage{amssymb}
\usepackage[dvips]{graphics}
\usepackage[margin=1in,includeheadfoot]{geometry}
\usepackage{url}
\usepackage{hyperref}
\usepackage{cite}

\def\beq{\begin{equation}}
\def\eeq{\end{equation}}
\def\ba{\begin{array}}
\def\ea{\end{array}}

\def\R{\mathbb R}

\newtheorem{thm}{Theorem}[section]
\newtheorem{lm}[thm]{Lemma}

\theoremstyle{definition}
\newtheorem{rem}[thm]{Remark}

\newtheorem{df}[thm]{Definition}

\newtheorem{assump}{Assumption}

\theoremstyle{remark}

\begin{document}
\pagestyle{plain}\today
\title{Boundary Lipschitz regularity of solutions for semilinear elliptic equations in divergence form}

\author{Jingqi Liang\\
 \small{School of Mathematical Sciences, Shanghai Jiao Tong University}\\
 \small{Shanghai, China}\\\\
\small{ Lihe Wang}\\
 \small{Department of Mathematics, University of Iowa, Iowa City, IA, USA; School of Mathematical Sciences, Shanghai Jiao Tong University}\\
 \small{Shanghai, China}\\\\
\small{Chunqin Zhou}\\
 \small{School of Mathematical Sciences, Shanghai Jiao Tong University}\\
 \small{Shanghai, China}}

\footnote{The third author was supported partially by NSFC of China (No. 11771285)}

\begin{abstract}

In this paper, we consider the pointwise boundary Lipschitz regularity of solutions for the semilinear elliptic equations in divergence form mainly under some weaker assumptions on nonhomogeneous term and the boundary. If the domain satisfies $C^{1,\text{Dini}}$ condition at a boundary point, and the nonhomogeneous term satisfies Dini continuous condition and Lipschitz Newtonian potential condition,  then the solution is Lipschitz continuous at this point. Furthermore, we generalize this result to Reifenberg $C^{1,\text{Dini}}$ domains.
\end{abstract}

\maketitle


\section{Introduction}
In this paper, we will investigate the boundary Lipschitz regularity of solutions for the following semilinear elliptic equation in divergence form:
\begin{equation}\label{1}
\left\{
\begin{array}{rcll}
\Delta u&=&\text {div}\overrightarrow{\mathbf{F}}(x,u)\qquad&\text{in}~~\Omega,\\
u&=&g\qquad&\text{on}~~\partial\Omega, \\
\end{array}
\right.
\end{equation}
where $\Omega$ is a bounded domain in $\mathbb{R}^{n}$.

There is a complete regularity theory for the classical Poisson equation
\begin{equation}\label{nond}\Delta u=f.
\end{equation}
 For the interior regularity of solutions, as we all know, $u$ is $C^{2,\alpha}$ for some $0<\alpha<1$ when  $f$ is $C^{\alpha}$, and  $u$ is $C^{2}$  when $f$ is Dini continuous, see \cite{DT, 7}.  Moreover, if $f\in L^\infty$, then any weak solution $u$ satisfies $u \in W^{2, p}_{loc}$ for any $1\leq p<+\infty$, and consequently $u\in C^{1,\alpha}$ for $0<\alpha <1$ but not for $\alpha =1$. So it is clear that $f\in L^{\infty}$ or even $f$ continuous is not strong enough to assure the $C^{1,1}$-regularity.  Recently research activity has thus focused on identifying conditions on $f$ which ensure $W^{2,\infty}_{loc}$ or $C^{1,1}_{loc}$ regularity of $u$.  In \cite{ALS}, Andersson, Lindgren and Shahgholian showed that  the sharp condition to get the $C^{1,1}$ regularity of $u$ is that $f*N$ is $C^{1,1}$  which is slightly weaker than the Dini condition,  where $N$ is the Newtonian potential and $*$ denotes the convolution.


For  the $C^{1,1}$ regularity of solutions for the semilinear Poisson equation
$$\Delta u=f(x,u)\quad \text{in}~~~B_{1},$$
which  are derived from the obstacle problem.
Shahgholian in \cite{6} proved $u$ is $C^{1,1}$ if $f(x,u)$ is Lipschitz continuous in $x$, uniformly in $u$, and $\partial_{u}f\geq-C$. Recently Indrei, Minne and Nurbekyan in \cite{1} obtained the same result under weaker assumptions that $f(x,u)$ is Dini continuous in $u$, uniformly in $x$, and it has a $C^{1,1}$ Newtonian potential in $x$, uniformly in $u$.

With respect to the boundary regularity of solutions, there has been extensive study in the past two decades which is closely related to the regularity of the boundary, such as sphere condition and $C^{1,\text{Dini}}$ condition. It is well-known that the solution of $(\ref{nond})$ is $C^{\alpha}$ up to the boundary when $\partial\Omega$ is Lipschitz. Trudinger in \cite{11, 12} proved the boundary Lipschitz regularity when the boundary satisfies uniform exterior sphere condition. Li and Wang in \cite{13, 14} studied the following Dirichlet problem,
\begin{equation}\label{aij}
\left\{
\begin{array}{rcll}
-a_{ij}(x)\frac{\partial^{2}u}{\partial x_{i}\partial x_{j}}&=&f(x) \qquad&\text{in}~~\Omega,\\
u&=&g\qquad&\text{on}~~\partial\Omega. \\
\end{array}
\right.
\end{equation}
They proved that the solution is differentiable at any point on the boundary  if the domain is convex. Li and Zhang in \cite{15} got the boundary differentiablity of (\ref{aij}) when $g=0$ under the $\gamma$-convexity domain condition which is strictly weaker than the convexity condition. On the other hand, the Dini continuity was a current topic for the regularity theory. In 2011, Ma and Wang \cite{3} considered the following equation,
\begin{equation}\label{fully}
\left\{
\begin{array}{rcll}
F(D^{2}u(x),x)&=&f(x) \qquad&\text{in}~~\Omega,\\
u&=&g\qquad&\text{on}~~\partial\Omega. \\
\end{array}
\right.
\end{equation}
They showed the pointwise $C^{1}$ estimates up to the boundary under the Dini conditions including that $\partial\Omega$ is $C^{1,\text{Dini}}$ (see Definition \ref{boundarydef}) and the boundary value $g$ is $C^{1,\text{Dini}}$. In \cite{4} Huang, Li and Wang proved that the solution of $(\ref{aij})$ is Lipschitz if the boundary satisfies exterior $C^{1,\text{Dini}}$ condition and the solution is differentiable if the boundary is exterior $C^{1}$ Dini and punctually $C^{1}$.  Furthermore, in \cite{5} they extended their results to Reifenberg $C^{1,\text{Dini}}$ domain which is more general than the classical $C^{1,\text{Dini}}$ domain.

However, there are few results on the boundary behavior of solutions for semilinear elliptic equations in the divergence form. In this paper, we study the pointwise boundary Lipschitz regularity of solutions for the semilinear elliptic equation in divergence form under some weaker assumptions on $\overrightarrow{\mathbf{F}}(x,u)$ and the boundary. For convenience we give some notations and definitions of Dini condition.

\
\

{\bf Notations:}

$|x|:=\sqrt{\sum\limits_{i=1}^{n} x_{i}^{2}}$: the Euclidean norm of $x=\left(x_{1}, x_{2}, \ldots, x_{n}\right) \in \mathbb{R}^{n}$.

$|x^{'}|:=\sqrt{\sum\limits_{i=1}^{n-1} x_{i}^{2}}$: the Euclidean norm of $x^{'}=\left(x_{1}, x_{2}, \ldots, x_{n-1}\right) \in \mathbb{R}^{n-1}$.

$B_{r}(x_{0}):=\left\{x \in \mathbb{R}^{n}:|x-x_{0}|<r\right\}$.

$B_{r}:=\left\{x \in \mathbb{R}^{n}:|x|<r\right\}$.

$\Omega_{r}:=B_{r}\cap\Omega$.

$T_{r}:=B_{r}\cap\{x_{n}=0\}=\left\{(x^{'},0) \in \mathbb{R}^{n}:|x^{'}|<r\right\}$.

$\vec{a}\cdot\vec{b}$: the standard inner product of $\vec{a}, \vec{b} \in \mathbb{R}^{n} .$

$\{\vec{e}_{i}\}_{i=1}^{n}$: the standard basis of $\mathbb{R}^{n}$.

\begin{df}\label{boundarydef}Let $x_{0}\in\partial\Omega.$ We say that $\partial\Omega$ is $C^{1,\text{Dini}}$ at $x_{0},$ if there exists a unit vector $\vec{n}$ and a positive constant $r_{0}$, a Dini modulus of continuity $\omega(r)$ satisfying $\int_{0}^{r_{0}} \frac{\omega(r)}{r} d r<\infty$ such that for any $0<r\leq r_{0},$
$$
B_{r}(x_0) \cap\left\{x\in\mathbb{R}^{n}:(x-x_{0}) \cdot \vec{n}>r \omega(r)\right\}\subset B_{r}(x_{0}) \cap \Omega \subset B_{r}(x_0) \cap\left\{x\in\mathbb{R}^{n}:(x-x_{0}) \cdot \vec{n}>-r \omega(r)\right\}.
$$
\end{df}
\begin{rem}
Any modulus of continuity $\omega(t)$ is non-decreasing, subadditive, continuous and satisfies $\omega(0)=0$ (see \cite{2}). Hence any modulus of continuity $\omega(t)$ satisfies
\begin{equation}\label{diniproperty}\frac{\omega(r)}{r}\leq2\frac{\omega(h)}{h},\quad 0<h<r.
\end{equation}
\end{rem}

\begin{df}\label{exrei}(Reifenberg $C^{1,\text{Dini}}$ condition). Let $x_{0}\in\partial\Omega.$ We say that $\Omega$ satisfies the $(r_{0}, \omega)$-Reifenberg $C^{1,\text{Dini}}$ condition at $x_{0}$ if there exists a positive constant $r_{0}$ and a Dini modulus of continuity $\omega(r)$ satisfying $\int_{0}^{r_{0}} \frac{\omega(r)}{r} d r<\infty$ such that for any $0<r\leq r_{0},$ there exists a unit vector $\vec{n}_{r}\in \mathbb{R}^{n}$ such that
$$
B_{r}(x_0) \cap\left\{x\in\mathbb{R}^{n}:(x-x_{0}) \cdot \vec{n}_{r}>r \omega(r)\right\}\subset B_{r}(x_{0}) \cap \Omega \subset B_{r}(x_0) \cap\left\{x\in\mathbb{R}^{n}:(x-x_{0}) \cdot \vec{n}_{r}>-r \omega(r)\right\}.
$$
\end{df}
The following lemma can be found in \cite{5}.
\begin{lm}\label{rem1.4}
If $\Omega$ satisfies $(r_{0}, \omega)$-Reifenberg $C^{1,\text{Dini}}$ condition, then there exists a bounded nonnegative function $S(\theta)\geq1$ such that $\left|\vec{n}_{r}-\vec{n}_{\theta r}\right| \leq S(\theta) \omega(r)$ for each $0<\theta<1$ and $0<r\leq r_{0}.$ Furthermore, for a fixed positive constant $0<\lambda<1,$ $\{\vec{n}_{\lambda^{i}r_{0}}\}_{i=0}^{\infty}$ is a Cauchy sequence. We can set $\lim\limits_{i\rightarrow\infty}\vec{n}_{\lambda^{i}r_{0}}=\vec{n}_{*}$.
\end{lm}
\begin{df}\label{bvdini}
Let $x_{0} \in \partial \Omega $.  The boundary value $g$ is said to be $C^{1,\text{Dini}}$ at $x_{0}$ with respect to a function $v_{x_0}(x)$,   if there exists a constant vector $\vec{a}$, a positive constant $r_{0}$ and a Dini modulus of continuity $\sigma(r)$ satisfying $\int_{0}^{r_{0}} \frac{\sigma(r)}{r} d r<\infty$ such that for any $0<r \leq r_{0}$ and $x \in \partial \Omega \cap B_{r}(x_{0})$,
$$
|g(x)-v_{x_0}(x)-g(x_{0})-v_{x_0}(x_{0})-\vec{a} \cdot(x-x_{0})| \leq r \sigma(r).
$$
\end{df}

Next we propose the following assumptions on $\overrightarrow{\mathbf{F}}(x,u)$.
\begin{assump}\label{as1} $\overrightarrow{\mathbf{F}}(x,t)\in L^{\infty}(B_{d}\times \R)$ where $d$ is large enough. Moreover
$\overrightarrow{\mathbf{F}}(x,t)$ is Dini continuous in $t$ with continuity modulus $\omega_{1}(r)$, uniformly in $x$, i.e.
$$\left|\overrightarrow{\mathbf{F}}(x,t_{2})-\overrightarrow{\mathbf{F}}(x,t_{1})\right|\leq\omega_{1}(|t_{2}-t_{1}|),
$$
and $\int_{0}^{t_{0}} \frac{\omega_{1}(t)}{t} dt<\infty$, for some $t_{0}>0.$
\end{assump}
\begin{assump}\label{as2}
For every boundary point $x_{0}$ and each $t\in \R$, there exists a function $v_{x_{0}}(\cdot,t)$ in $B_{1}(x_{0})$ satisfying
$$
\Delta v_{x_{0}}(\cdot, t)=\text{div}\overrightarrow{\mathbf{F}}(\cdot,t) \quad \text{in}~~B_{1}(x_{0}).
$$
Furthermore, $v_{x_{0}}(\cdot,t)$  is a Lipschitz function which is  uniform in $x_0$ and $t$ with Lipschitz constant $T$.
\end{assump}
\begin{rem}
(1) We can always assume that $\Omega\subset B_{\frac{d}{2}}$.\\
(2) In the sequel, for $u(x)\in L^{\infty}(\Omega)$, for every boundary point $x_{0}$, we let $v_{x_{0}}(x)$ solve
$$\Delta v_{x_{0}}(x)=\text{div}\overrightarrow{\mathbf{F}}(x,u(x_{0})) \quad \text{in}~~B_{1}(x_{0}).$$
\end{rem}
Our main results are the following two theorems.
\begin{thm}\label{mr1}
Let $x_{0} \in \partial \Omega .$ $\overrightarrow{\mathbf{F}}(x,t)$ satisfies Assumption \ref{as1} and \ref{as2}. If $\partial\Omega$ is $C^{1,\text{Dini}}$ at $x_{0}$ and $g$ is $C^{1,\text{Dini}}$ at $x_{0}$ with respect to $v_{x_{0}}(x)$,
then the solution of $(\ref{1})$ is Lipschitz continuous at $x_{0}$, i.e.
$$|u(x)-u(x_{0})|\leq C|x-x_{0}|,\quad \forall~x\in B_{1}\cap\Omega.$$
where $C=C(n, \Omega, T, \|u\|_{L^{\infty}(\Omega)}, \|g\|_{L^{\infty}(\partial\Omega)})$.
\end{thm}
\begin{thm}\label{mr2}
Let $x_{0} \in \partial \Omega .$ $\overrightarrow{\mathbf{F}}(x,t)$ satisfies Assumption \ref{as1} and \ref{as2}. If $\partial\Omega$ satisfies Reifenberg $C^{1,\text{Dini}}$ condition at $x_{0}$ and $g$ is $C^{1,\text{Dini}}$ at $x_{0}$ with respect to $v_{x_{0}}(x)$,
then the solution of $(\ref{1})$ is Lipschitz continuous at $x_{0}$.
\end{thm}

To prove these two theorems, the main method is iterative scheme. In fact we can approximate $u$ by $v$ defined by Assumption \ref{as2} and to show this approximation can be improved from $B_{1}$ to $B_{\lambda}$. The key point of our proof is that the main part of $u$ is a Lipschitz function $v$ and a linear function.  We organize the paper as follows. In section 2 we will give some necessary lemmas. Next we prove the Lipschitz regularity under $C^{1,\text{Dini}}$ condition in Section 3. Finally, we extend $C^{1,\text{Dini}}$ condition to Reifenberg $C^{1,\text{Dini}}$ condition and give a proof of Theorem \ref{mr2} in Section 4.

\section{Preliminary tools}
In this section we will give a general approximation lemma of the following divergence form elliptic equation:
\begin{equation}\label{2}
\Delta u=\text{div}\overrightarrow{\mathbf{F}} \quad \text{in}~~\Omega,
\end{equation}
where $\Omega$ is a bounded domain. After proving the approximation lemma, we will give a key lemma which will be used repeatedly in Section 3 and 4. We mainly assume that the boundary lies  between two parallel hyperplanes with a very small distance.
\begin{lm}\label{fs}
$0\in\partial\Omega,$ $B_{1}\cap\{x\in\mathbb{R}^{n}: x_{n}>\varepsilon\}\subset B_{1}\cap\Omega\subset B_{1}\cap\{x\in\mathbb{R}^{n}: x_{n}>-\varepsilon\}$ for some $0<\varepsilon<\frac{1}{4}$. For any $\overrightarrow{\mathbf{F}}(x)\in L^{\infty}(B_{1}\cap\Omega),~g\in L^{\infty}(B_{1}\cap\partial\Omega),$ if $u$ is a weak solution of
\begin{equation}\label{3}
\left\{
\begin{array}{rcll}
\Delta u&=&\text{div}\overrightarrow{\mathbf{F}}\qquad&\text{in}~~B_{1}\cap\Omega,\\
u&=&g\qquad&\text{on}~~B_{1}\cap\partial\Omega, \\
\end{array}
\right.
\end{equation}
then there exists a universal constant $C_{0}$ and a harmonic function $h$ defined in $B_{\frac{1}{4}}$ which is odd with respect to $x_{n}$ satisfying
$$\|h\|_{L^{\infty}(B_{\frac{1}{4}})}\leq(1+2C_{0}\varepsilon)\|u\|_{L^{\infty}(B_{1}\cap\Omega)}$$
such that
$$\|u-h\|_{L^{\infty}(B_{\frac{1}{4}}\cap\Omega)}\leq \|g\|_{L^{\infty}(B_{1}\cap\partial\Omega)}+5C_{0}\varepsilon\|u\|_{L^{\infty}(B_{1}\cap\Omega)}+C\|\overrightarrow{\mathbf{F}}\|_{L^{\infty}(B_{1}\cap\Omega)},
$$
where $C$ is a constant depending only on $n$ and $\Omega$.
\end{lm}

In order to prove this lemma, we need the following estimate which can be found in many basic books of elliptic partial differential equations, such as Chapter 8 in \cite{DT}.
\begin{thm}\label{lemma1}
Let $\Omega$ be a bounded domain in $\mathbb{R}^{n}.$ We suppose that $\overrightarrow{\mathbf{F}}\in L^{q}(\Omega),$ for some $q>n.$ Then if $u$ is a $W^{1,2}(\Omega)$ solution of (\ref{2}), then we have
$$\|u\|_{L^{\infty}(\Omega)}\leq\|u\|_{L^{\infty}(\partial\Omega)}+C\|\overrightarrow{\mathbf{F}}\|_{L^{q}(\Omega)},
$$
where $C=C(n,q,|\Omega|).$
\end{thm}

Now we give the proof of Lemma \ref{fs}.
\begin{proof}
Let $v$ solve the following problem
\begin{equation}\label{4}
\left\{
\begin{array}{rcll}
\Delta v&=&0\qquad&\text{in}~~B_{1}\cap\Omega,\\
v&=&u\qquad&\text{on}~~\partial B_{1}\cap\Omega, \\
v&=&0\qquad&\text{on}~~B_{1}\cap\partial\Omega.\\
\end{array}
\right.
\end{equation}
Then by the maximum principle, we have
$$|v|\leq\|u\|_{L^{\infty}(B_{1}\cap\Omega)}\quad \text{in}~~B_{1}\cap\Omega.
$$
Next we denote $\|u\|_{L^{\infty}(B_{1}\cap\Omega)}$ by $\mu$. Combining with (\ref{3}) and (\ref{4}) we have
\begin{eqnarray*}
\left\{
\begin{array}{rcll}
\Delta (u-v)&=&\text{div}\overrightarrow{\mathbf{F}}\qquad&\text{in}~~B_{1}\cap\Omega,\\
u-v&=&0\qquad&\text{on}~~\partial B_{1}\cap\Omega, \\
u-v&=&g\qquad&\text{on}~~B_{1}\cap\partial\Omega.\\
\end{array}
\right.
\end{eqnarray*}
By the previous lemma we obtain
\begin{equation}\label{step1}
\|u-v\|_{L^{\infty}(B_{\frac{1}{2}}\cap\Omega)}\leq\|g\|_{L^{\infty}(B_{1}\cap\partial\Omega)}+C\|\overrightarrow{\mathbf{F}}\|_{L^{\infty}(B_{1}\cap\Omega)}.
\end{equation}
Let $\Gamma$ be defined for $x\in\mathbb{R}^{n}\setminus\{0\}$ by
\begin{eqnarray*} \Gamma(x)=\Gamma(|x|)=
\begin{cases}
-\frac{1}{2\pi}\ln|x|, \quad &n=2,\\
\frac{1}{(n-2)\omega_{n}}|x|^{2-n},\quad &n\geq3,\\
\end{cases}
\end{eqnarray*}
where $\omega_{n}$ is the surface area of the unit sphere in $\mathbb{R}^{n}$. This function $\Gamma$ is usually called the fundamental solution of the Laplace operator. By a simple calculation, we have $\Delta\Gamma=0$ in $\mathbb{R}^{n}\setminus\{0\}$. Then for any $y=(y',0)\in T_{\frac{1}{4}}$, we consider a function
$$l(x)=\frac{\Gamma\left(x-\left(y',-\frac{1}{4}-\varepsilon\right)\right)-\Gamma\left(\frac{1}{4}\right)}{\Gamma\left(\frac{3}{4}\right)-\Gamma\left(\frac{1}{4}\right)}\mu.
$$
Clearly $l(x)$ is harmonic between $B_{\frac{1}{4}}(y',-\frac{1}{4}-\varepsilon)$ and $B_{\frac{3}{4}}(y',-\frac{1}{4}-\varepsilon)$ and
\begin{equation}\label{5}
\left\{
\begin{array}{rcll}
&l(x)&=0\qquad&\text{on}~~\partial B_{\frac{1}{4}}(y',-\frac{1}{4}-\varepsilon),\\
0<&l(x)&<\mu\qquad&\text{between}~B_{\frac{1}{4}}(y',-\frac{1}{4}-\varepsilon)~\text{and}~B_{\frac{3}{4}}(y',-\frac{1}{4}-\varepsilon), \\
&l(x)&=\mu\qquad&\text{on}~~\partial B_{\frac{3}{4}}(y',-\frac{1}{4}-\varepsilon).\\
\end{array}
\right.
\end{equation}
From (\ref{4}) and (\ref{5}) we get $v-l$ satisfies
\begin{eqnarray*}
\left\{
\begin{array}{rcll}
\Delta (v-l)&=&0\qquad&\text{in}~~B_{\frac{3}{4}}(y',-\frac{1}{4}-\varepsilon)\cap\Omega,\\
v-l&\leq&0\qquad&\text{on}~~\partial B_{\frac{3}{4}}(y',-\frac{1}{4}-\varepsilon)\cap\Omega, \\
v-l&\leq&0\qquad&\text{on}~~B_{\frac{3}{4}}(y',-\frac{1}{4}-\varepsilon)\cap\partial\Omega.\\
\end{array}
\right.
\end{eqnarray*}
Applying the maximum principle again, it yields that
$$v\leq l \quad \text{in}~~B_{\frac{3}{4}}(y',-\frac{1}{4}-\varepsilon)\cap\Omega,
$$
similarly, repeat the above process for $v+l$, it's easy to see that $v\geq-l~ \text{in}~B_{\frac{3}{4}}(y',-\frac{1}{4}-\varepsilon)\cap\Omega,
$
then
$$|v|\leq l \quad \text{in}~~B_{\frac{3}{4}}(y',-\frac{1}{4}-\varepsilon)\cap\Omega.
$$
Furthermore for arbitrary $x_{1}\in\partial B_{\frac{1}{4}}(y',-\frac{1}{4}-\varepsilon)$, in the radial direction, we have
$$\frac{l(x)-l(x_{1})}{|x-x_{1}|}\leq C_{0}\mu \quad \text{between}~B_{\frac{1}{4}}(y',-\frac{1}{4}-\varepsilon)~\text{and}~B_{\frac{3}{4}}(y',-\frac{1}{4}-\varepsilon),
$$
then
$$l(x)\leq C_{0}\mu \text{dist}(x,\partial B_{\frac{1}{4}}(y',-\frac{1}{4}-\varepsilon)).
$$
In particular along the $x_{n}-$direction, we get $l(x)\leq C_{0}\mu(x_{n}+\varepsilon)$. Since $y$ can be chosen in $T_{\frac{1}{4}}$ arbitrarily, then
\begin{equation}\label{6}
|v(x)|\leq l(x)\leq C_{0}\mu(x_{n}+\varepsilon) \quad \text{in}~\overline{B_{\frac{1}{4}}}\cap\Omega.
\end{equation}
Then (\ref{4}) and (\ref{6}) imply that $v$ satisfies the following conditions,
\begin{eqnarray*}
\left\{
\begin{array}{rcll}
\Delta v&=&0\qquad&\text{in}~~B_{\frac{1}{4}}\cap\Omega,\\
|v(x)|&\leq& C_{0}\mu(x_{n}+\varepsilon)\qquad&\text{in}~~ \overline{B_{\frac{1}{4}}}\cap\Omega, \\
v&=&0\qquad&\text{on}~~B_{\frac{1}{4}}\cap\partial\Omega.\\
\end{array}
\right.
\end{eqnarray*}
Now it's time to find the harmonic function. We take $h$ be a harmonic function defined in $B_{\frac{1}{4}}$ which is odd with respect to $x_{n}$ and satisfies the following conditions,
\begin{eqnarray*}
\left\{
\begin{array}{rcll}
\Delta h&=&0\qquad&\text{in}~~B_{\frac{1}{4}}^{+},\\
h&=&0\qquad&\text{on}~~T_{\frac{1}{4}}, \\
h&=&v\qquad&\text{on}~~\partial B_{\frac{1}{4}}^{+}\cap\{x\in\mathbb{R}^{n}: x_{n}\geq\varepsilon\},\\
h&=&2C_{0}\mu\varepsilon\qquad&\text{on}~~\partial B_{\frac{1}{4}}^{+}\cap\{x\in\mathbb{R}^{n}: 0<x_{n}<\varepsilon\}.\\
\end{array}
\right.
\end{eqnarray*}
Applying the maximum principle to $h$, we get
\begin{equation}\label{7}|h(x)|\leq C_{0}\mu(x_{n}+\varepsilon)+2C_{0}\mu\varepsilon \quad \text{in}~~~B_{\frac{1}{4}}^{+},
\end{equation}
$$\|h\|_{L^{\infty}(B_{\frac{1}{4}}^{+})}\leq(1+2C_{0}\varepsilon)\mu.$$
Next we consider $v-h$ in $B_{\frac{1}{4}}^{+}\cap\Omega$ to obtain
\begin{eqnarray*}
\left\{
\begin{array}{rcll}
&\Delta (v-h)&=0\qquad&\text{in}~~B_{\frac{1}{4}}^{+}\cap\Omega,\\
&v-h&=0\qquad&\text{on}~~\partial B_{\frac{1}{4}}^{+}\cap\{x_{n}\geq\varepsilon\},\\
-4C_{0}\mu\varepsilon\leq &v-h&\leq 0\qquad&\text{on}~~\partial B_{\frac{1}{4}}^{+}\cap\{0<x_{n}<\varepsilon\},\\
-4C_{0}\mu\varepsilon\leq&v-h&\leq4C_{0}\mu\varepsilon\qquad&\text{on}~~B_{\frac{1}{4}}^{+}\cap\partial\Omega,\\
-C_{0}\mu\varepsilon\leq&v-h&\leq C_{0}\mu\varepsilon\qquad&\text{on}~~T_{\frac{1}{4}}\cap\Omega. \\
\end{array}
\right.
\end{eqnarray*}
Using the maximum principle again we obtain
$$|v-h|\leq4C_{0}\mu\varepsilon\quad\text{in}~~B_{\frac{1}{4}}^{+}\cap\Omega.$$
Since $h$ is odd respect to $x_{n}$ and $B_{1}\cap\Omega\subset B_{1}\cap\{x\in\mathbb{R}^{n}: x_{n}>-\varepsilon\}$ for some $0<\varepsilon<\frac{1}{4}$, it's easy to get $|h|\leq4C_{0}\mu\varepsilon$ in $B_{\frac{1}{4}}^{-}\cap\Omega.$
Combining with (\ref{6}) we get
$$|v-h|\leq5C_{0}\mu\varepsilon\quad \text{in}~~B_{\frac{1}{4}}^{-}\cap\Omega.$$
From above two inequalities we get
\begin{equation}\label{10}\left\|v-h\right\|_{L^{\infty}(B_{\frac{1}{4}}\cap\Omega)}\leq5C_{0}\mu\varepsilon.
\end{equation}
Then from (\ref{step1}) and (\ref{10}), we can get the following desired result by the triangle inequality,
$$\left\|u-h\right\|_{L^{\infty}(B_{\frac{1}{4}}\cap\Omega)}\leq5C_{0}\mu\varepsilon+\|g\|_{L^{\infty}(B_{1}\cap\partial\Omega)}+C\|\overrightarrow{\mathbf{F}}\|_{L^{\infty}(B_{1}\cap\Omega)}.
$$
\end{proof}
\begin{rem}\label{rem2.3}
$x_{n}$ can be regarded as $x\cdot\vec{e}_{n}$. Therefore in the above lemma, $x_{n}$ can be replaced by $x\cdot\vec{n}$ for arbitrary unit vector $\vec{n}$.
\end{rem}
Then we give a more general form of Lemma $\ref{fs}$.
\begin{lm}\label{fs2}
$0\in\partial\Omega,$ $B_{1}\cap\{x\in\mathbb{R}^{n}: x\cdot\vec{n}>\varepsilon\}\subset B_{1}\cap\Omega\subset B_{1}\cap\{x\in\mathbb{R}^{n}: x\cdot\vec{n}>-\varepsilon\}$ for some $0<\varepsilon<\frac{1}{4}$. For any $\overrightarrow{\mathbf{F}}(x)\in L^{\infty}(B_{1}\cap\Omega),~g\in L^{\infty}(B_{1}\cap\partial\Omega),$ if $u$ is a weak solution of
\begin{eqnarray*}
\left\{
\begin{array}{rcll}
\Delta u&=&\text{div}\overrightarrow{\mathbf{F}}\qquad&\text{in}~~B_{1}\cap\Omega,\\
u&=&g\qquad&\text{on}~~B_{1}\cap\partial\Omega, \\
\end{array}
\right.
\end{eqnarray*}
then there exists a universal constant $C_{0}$ and a harmonic function $h$ defined in $B_{\frac{1}{4}}$ which is odd with respect to $x\cdot\vec{n}$ satisfying
$$\|h\|_{L^{\infty}(B_{\frac{1}{4}})}\leq(1+2C_{0}\varepsilon)\|u\|_{L^{\infty}(B_{1}\cap\Omega)}$$
such that
$$\|u-h\|_{L^{\infty}(B_{\frac{1}{4}}\cap\Omega)}\leq \|g\|_{L^{\infty}(B_{1}\cap\partial\Omega)}+5C_{0}\varepsilon\|u\|_{L^{\infty}(B_{1}\cap\Omega)}+C\|\overrightarrow{\mathbf{F}}\|_{L^{\infty}(B_{1}\cap\Omega)},
$$
where $C$ is a constant depending only on $n$ and $\Omega$.
\end{lm}
\begin{lm}[Key lemma]\label{kl}
Let $\Omega$ be a bounded domain in $\mathbb{R}^{n}$. Assume that $0\in\partial\Omega$ and $B_{1}\cap\{x\in\mathbb{R}^{n}: x_{n}>\varepsilon\}\subset B_{1}\cap\Omega\subset B_{1}\cap\{x\in\mathbb{R}^{n}: x_{n}>-\varepsilon\}$ for some $0<\varepsilon<\frac{1}{4}$. Then there exists $\lambda>0$ and universal constants $\widetilde{C}, C_{1}, C_{2}>0$ such that for any functions $\overrightarrow{\mathbf{F}}(x,u)\in L^{\infty}(B_{1}\cap\Omega)$, $g\in L^{\infty}(B_{1}\cap\partial\Omega)$, the solution of
$$
\left\{
\begin{array}{rcll}
\Delta u&=&\text{div}\overrightarrow{\mathbf{F}}(x,u)\qquad&\text{in}~~B_{1}\cap\Omega,\\
u&=&g\qquad&\text{on}~~B_{1}\cap\partial\Omega, \\
\end{array}
\right.
$$
and solution of
$$
\Delta v=\text{div}\overrightarrow{\mathbf{F}}(x,0)\quad \text{in}~~B_{1},
$$
there exists a constant K such that
\begin{eqnarray*}
\|u-v-Kx_{n}\|_{L^{\infty}(B_{\lambda}\cap\Omega)}&\leq& \|g-v\|_{L^{\infty}(B_{1}\cap\partial\Omega)}
+C_{1}(\lambda^{2}+\varepsilon)\|u-v\|_{L^{\infty}(B_{1}\cap\Omega)}\\
&&+C_{2}\|\overrightarrow{\mathbf{F}}(x,u)-\overrightarrow{\mathbf{F}}(x,0)\|_{L^{\infty}(B_{1}\cap\Omega)},
\end{eqnarray*}
and
$$0<K\leq\widetilde{C}\|u-v\|_{L^{\infty}(B_{1}\cap\Omega)}.
$$
\end{lm}
\begin{proof}
By the definition of $u$ and $v$ we get
$$
\left\{
\begin{array}{rcll}
\Delta (u-v)&=&\text{div}(\overrightarrow{\mathbf{F}}(x,u)-\overrightarrow{\mathbf{F}}(x,0))\qquad&\text{in}~~B_{1}\cap\Omega,\\
u-v&=&g-v\qquad&\text{on}~~B_{1}\cap\partial\Omega. \\
\end{array}
\right.
$$
Then by Lemma $\ref{fs}$, there exists a universal constant $C_{0}$ and a harmonic function $h$ defined in $B_{\frac{1}{4}}$ which is odd with respect to $x_{n}$ satisfying
$$\|h\|_{L^{\infty}(B_{\frac{1}{4}})}\leq(1+2C_{0}\varepsilon)\|u\|_{L^{\infty}(B_{1}\cap\Omega)}$$
such that
\begin{equation}\label{11}
\|u-v-h\|_{L^{\infty}(B_{\frac{1}{4}}\cap\Omega)}\leq \|g-v\|_{L^{\infty}(B_{1}\cap\partial\Omega)}
+5C_{0}\varepsilon\|u-v\|_{L^{\infty}(B_{1}\cap\Omega)}+C\|\overrightarrow{\mathbf{F}}(x,u)-\overrightarrow{\mathbf{F}}(x,0)\|_{L^{\infty}(B_{1}\cap\Omega)}.\\
\end{equation}

Take $L$ be the first order Taylor polynomial of $h$ at 0, i.e. $L(x)=Dh(0)\cdot x+h(0).$
Since $h=0$ on $B_{\frac{1}{4}}\cap\{x_{n}=0\},$ then $L(x)=|Dh(0)|x_{n}.
$ Note that $h$ is a harmonic function which is odd with respect to $x_{n}$ in $B_{\frac{1}{4}}$, according to the property of harmonic function, when $|x|\leq\frac{1}{8},$
$$|D^{2}h(x)|+|Dh(x)|\leq A\|h\|_{L^{\infty}(B_{\frac{1}{4}})}\leq A(1+2C_{0}\varepsilon)\|u-v\|_{L^{\infty}(B_{1}\cap\Omega)},
$$
where $A$ is a constant depending only on $n$. Then there exists $\xi\in B_{\frac{1}{8}}$ such that for $|x|\leq\frac{1}{8},$
\begin{equation}\label{12}
|h(x)-L(x)|\leq\frac{1}{2}|D^{2}h(\xi)||x|^{2}.
\end{equation}

Finally, combining with (\ref{11}) and (\ref{12}), if we set $K=|Dh(0)|$ and take $0<\lambda<\frac{1}{8},$ we have
\begin{eqnarray*}
\|u-v-Kx_{n}\|_{L^{\infty}(B_{\lambda}\cap\Omega)}&\leq&\|u-v-h\|_{L^{\infty}(B_{\lambda}\cap\Omega)}+\|h-L\|_{L^{\infty}(B_{\lambda}\cap\Omega)}\\
&\leq&\|g-v\|_{L^{\infty}(B_{1}\cap\partial\Omega)}+5C_{0}\varepsilon\|u-v\|_{L^{\infty}(B_{1}\cap\Omega)}+C\|\overrightarrow{\mathbf{F}}(x,u)-\overrightarrow{\mathbf{F}}(x,0)\|_{L^{\infty}(B_{1}\cap\Omega)}\\
&&+\frac{1}{2}\lambda^{2}A(1+2C_{0}\varepsilon)\|u-v\|_{L^{\infty}(B_{1}\cap\Omega)}\\
&\leq&\|g-v\|_{L^{\infty}(B_{1}\cap\partial\Omega)}+C_{1}(\lambda^{2}+\varepsilon)\|u-v\|_{L^{\infty}(B_{1}\cap\Omega)}\\
&&+C_{2}\|\overrightarrow{\mathbf{F}}(x,u)-\overrightarrow{\mathbf{F}}(x,0)\|_{L^{\infty}(B_{1}\cap\Omega)}.
\end{eqnarray*}
where $C_{1}=\max\{5C_{0},\frac{1}{2}A(1+2C_{0})\}$, $C_{2}=C.$
\end{proof}
\begin{rem}\label{rem3.1}
Let $\Omega$ be a bounded domain in $\mathbb{R}^{n}$. Assume that $0\in\partial\Omega$ and $B_{1}\cap\{x\in\mathbb{R}^{n}: x_{n}>\varepsilon\}\subset B_{1}\cap\Omega\subset B_{1}\cap\{x\in\mathbb{R}^{n}: x_{n}>-\varepsilon\}$ for some $0<\varepsilon<\frac{1}{4}$. Then there exists $\lambda>0$ and universal constants $\widetilde{C}, C_{1}, C_{2}>0$ such that for any functions $\overrightarrow{\mathbf{F}}(x)\in L^{\infty}(B_{1}\cap\Omega)$, $g\in L^{\infty}(B_{1}\cap\partial\Omega)$, the solution of
$$
\left\{
\begin{array}{rcll}
\Delta u&=&div\overrightarrow{\mathbf{F}}(x)\qquad&\text{in}~~B_{1}\cap\Omega,\\
u&=&g\qquad&\text{on}~~B_{1}\cap\partial\Omega, \\
\end{array}
\right.
$$
there exists a constant K such that
\begin{eqnarray*}
\|u-Kx_{n}\|_{L^{\infty}(B_{\lambda}\cap\Omega)}&\leq& \|g\|_{L^{\infty}(B_{1}\cap\partial\Omega)}
+C_{1}(\lambda^{2}+\varepsilon)\|u\|_{L^{\infty}(B_{1}\cap\Omega)}\\
&&+C_{2}\|\overrightarrow{\mathbf{F}}(x)\|_{L^{\infty}(B_{1}\cap\Omega)}
\end{eqnarray*}
and
$$0<K\leq\widetilde{C}\|u\|_{L^{\infty}(B_{1}\cap\Omega)}.
$$
\end{rem}
\begin{rem}\label{rem2.7}
Similar to Remark $\ref{rem2.3}$, $x_{n}$ can also be regarded as $x\cdot\vec{e}_{n}$ and can be substituted by $x\cdot\vec{n}$ for arbitrary unit vector $\vec{n}$ in Lemma $\ref{kl}$ and Remark $\ref{rem3.1}$.
\end{rem}

\section{Boundary Lipschitz regularity under $C^{1,\text{Dini}}$ condition}
In this section we will give a proof of Theorem \ref{mr1}. The following four lemmas lead to the desired result. The first step has been finished in Section 2, i.e. we got the key lemma. Next we aim to iterate step by step and get the approximations in different scales. Finally we will prove that the sum of errors from different scales is convergent. In addition, Dini conditions play a great role to obtain the convergence. For convenience, we can choose an appropriate coordinate system such that $\vec{n}$ in Definition $\ref{boundarydef}$ is along the positive $x_{n}$-axis in the proof. So by definition, if $\partial\Omega$ is $C^{1,\text{Dini}}$ at $x_{0},$ then for any $0<r\leq r_{0},$ $B_{r}(x_{0})\cap\partial\Omega\subset B_{r}(x_{0})\cap\left\{|x_{n}-x_{0,n}|\leq r\omega(r)\right\}$. Without loss of generality, we can take $x_{0}=0$, denote $v_{0}$ by $v$ and assume that
$$\quad u(0)=g(0)=0,\quad v(0)=0, \quad r_{0}=1,
$$
$$\omega(1)\leq\lambda, \quad \int_{0}^{1} \frac{\omega(r)}{r} d r \leq 1, \quad \int_{0}^{1} \frac{\sigma(r)}{r} d r \leq 1, \quad \int_{0}^{1} \frac{\omega_{1}(r)}{r} d r \leq 1,
$$
where $\lambda$ will be determined in Lemma $\ref{kl}$ and $\ref{con}$. Besides, We can also assume $\vec{a}=0$ in definition \ref{bvdini}, if not, we can consider $\widehat{u}:=u-g(0)-\vec{a}\cdot x,$ which satisfies the same equation.

Based on Lemma $\ref{kl}$ and Remark $\ref{rem3.1}$, the following lemma is an iteration result.
\begin{lm}\label{bmkm}
There exist nonnegative sequences $\{M_{i}\}_{i=0}^{\infty},$ and $\{N_{i}\}_{i=0}^{\infty},$ with $N_{0}=0$, $M_{0}=\|u-v\|_{L^{\infty}(\Omega_{1})}$, and for $i=0,1,2,\ldots,$
\begin{eqnarray*}
M_{i+1}&=&\|g-v-N_{i}x_{n}\|_{L^{\infty}(B_{\lambda^{i}}\cap\partial\Omega)}\\
&&+C_{1}(\lambda^{2}+\omega(\lambda^{i}))\|u-v-N_{i}x_{n}\|_{L^{\infty}(\Omega_{\lambda^{i}})}\\
&&+C_{2}\lambda^{i}\|\overrightarrow{\mathbf{F}}(x,u)-\overrightarrow{\mathbf{F}}(x,0)\|_{L^{\infty}(\Omega_{\lambda^{i}})},
\end{eqnarray*}
\begin{eqnarray*}
\label{n}|N_{i+1}-N_{i}|\leq\frac{\widetilde{C}}{\lambda^{i}}\|u-v-N_{i}x_{n}\|_{L^{\infty}(\Omega_{\lambda^{i}})},
\end{eqnarray*}
such that
\begin{equation}\label{induction}
\|u-v-N_{i}x_{n}\|_{L^{\infty}(\Omega_{\lambda^{i}})}\leq M_{i}.
\end{equation}
\end{lm}
\begin{proof}
We prove this lemma inductively by using Remark $\ref{rem3.1}$ repeatedly.

When $i=0$, since $N_{0}=0$ and $M_{0}=\|u-v\|_{L^{\infty}(\Omega_{1})}$,  it's easy to see
$$\|u-v-N_{0}x_{n}\|_{L^{\infty}(\Omega_{1})}= M_{0}.$$

When $i=1,$ by Definition $\ref{boundarydef}$, we have $B_{1}\cap\partial\Omega\subset\left\{|x_{n}|\leq \omega(1)\right\}$. Therefore by lemma $\ref{kl}$, there exists $0<N_{1}\leq\widetilde{C}\|u-v\|_{L^{\infty}(B_{1}\cap\Omega)}$ such that
\begin{eqnarray*}
\|u-v-N_{1}x_{n}\|_{L^{\infty}(B_{\lambda}\cap\Omega)}&\leq&
\|g-v\|_{L^{\infty}(B_{1}\cap\partial\Omega)}+C_{1}(\lambda^{2}+\omega(1))\|u-v\|_{L^{\infty}(B_{1}\cap\Omega)}\\
&&+C_{2}\|\overrightarrow{\mathbf{F}}(x,u)-\overrightarrow{\mathbf{F}}(x,0)\|_{L^{\infty}(B_{1}\cap\Omega)}\triangleq M_{1},
\end{eqnarray*}
and
$$|N_{1}-N_{0}|\leq\widetilde{C}\|u-v\|_{L^{\infty}(B_{1}\cap\Omega)}.$$
Next we assume that the conclusion is true for $i$. We consider the equation
$$
\left\{
\begin{array}{rcll}
\Delta (u-v-N_{i}x_{n})&=&\text{div}(\overrightarrow{\mathbf{F}}(x,u)-\overrightarrow{\mathbf{F}}(x,0))\qquad&\text{in}~~\Omega_{\lambda^{i}},\\
u-v-N_{i}x_{n}&=&g-v-N_{i}x_{n}\qquad&\text{on}~~B_{\lambda^{i}}\cap\partial\Omega. \\
\end{array}
\right.
$$
For $z=(z_{1}, z_{2}, \cdots,z_{n})\in\mathbb{R}^{n}$ we set
$$\widetilde{u}(z)=\frac{u(\lambda^{i} z)-v(\lambda^{i} z)-N_{i}\lambda^{i} z_{n}}{\lambda^{i}},
$$
$$\widetilde{g}(z)=\frac{g(\lambda^{i} z)-v(\lambda^{i} z)-N_{i}\lambda^{i} z_{n}}{\lambda^{i}},
$$
$$\widetilde{f}(z)=\overrightarrow{\mathbf{F}}(\lambda^{i} z,u(\lambda^{i} z))-\overrightarrow{\mathbf{F}}(\lambda^{i} z,0).
$$
Then $\widetilde{u}(z)$ is a solution of
$$
\left\{
\begin{array}{rcll}
\Delta \widetilde{u}(z)&=&\text{div}{\widetilde{f}}(z)\qquad&\text{in}~~B_{1}\cap\widetilde{\Omega},\\
\widetilde{u}(z)&=&\widetilde{g}(z)\qquad&\text{on}~~B_{1}\cap\partial\widetilde{\Omega}, \\
\end{array}
\right.
$$
where $\widetilde{\Omega}=\{z:\lambda^{i}z\in\Omega\}.$ Therefore $B_{1}\cap\partial\widetilde{\Omega}\subset B_{1}\cap\left\{|z_{n}|\leq \omega(\lambda^{i})\right\}$. Then by Remark $\ref{rem3.1}$, there exists a constant $K$ such that
\begin{eqnarray*}
\|\widetilde{u}-Kz_{n}\|_{L^{\infty}(B_{\lambda}\cap\widetilde{\Omega})}
&\leq&\|\widetilde{g}\|_{L^{\infty}(B_{1}\cap\partial\widetilde{\Omega})}+C_{1}(\lambda^{2}+\omega(\lambda^{i}))\|\widetilde{u}\|_{L^{\infty}(B_{1}\cap\widetilde{\Omega})}\\
&&+C_{2}\|{\widetilde{f}}\|_{L^{\infty}(B_{1}\cap\widetilde{\Omega})},
\end{eqnarray*}
where $K\leq \widetilde{C}\|\widetilde{u}\|_{L^{\infty}(B_{1}\cap\widetilde{\Omega})}=\frac{\widetilde{C}}{\lambda^{i}}\|u-v-N_{i}x_{n}\|_{L^{\infty}(\Omega_{\lambda^{i}})}.$
Let $N_{i+1}=N_{i}+K,$ scaling back, then we get
\begin{eqnarray*}
\|u-v-N_{i+1}x_{n}\|_{L^{\infty}(\Omega_{\lambda^{i+1}})}
&\leq&\|g-v-N_{i}x_{n}\|_{L^{\infty}(B_{\lambda^{i}}\cap\partial\Omega)}\\
&&+C_{1}(\lambda^{2}+\omega(\lambda^{i}))\|u-v-N_{i}x_{n}\|_{L^{\infty}(\Omega_{\lambda^{i}})}\\
&&+C_{2}\lambda^{i}\|\overrightarrow{\mathbf{F}}(x,u)-\overrightarrow{\mathbf{F}}(x,0)\|_{L^{\infty}(\Omega_{\lambda^{i}})}\triangleq M_{i+1},
\end{eqnarray*}
and
$$|N_{i+1}-N_{i}|=K\leq\frac{\widetilde{C}}{\lambda^{i}}\|u-v-N_{i}x_{n}\|_{L^{\infty}(\Omega_{\lambda^{i}})}.$$
This completes the proof of Lemma $\ref{bmkm}$.
\end{proof}

The following three lemmas are similar to \cite{4} and \cite{5}.
\begin{lm}\label{con}
$\sum\limits_{i=0}^{\infty} \frac{M_{i}}{\lambda^{i}}<\infty$ and $\lim \limits_{i \rightarrow \infty} N_{i}$ exists. We set
$$
\lim _{i \rightarrow \infty} N_{i}=\tau.
$$
\end{lm}
\begin{proof}
We assume $T$ is the Lipschitz constant respect to $v$, then
$$\|v\|_{L^{\infty}(B_{\lambda^{i}}\cap\partial\Omega)}=\|v-v(0)\|_{L^{\infty}(B_{\lambda^{i}}\cap\partial\Omega)}\leq T\lambda^{i}.
$$
For $k\geq0,$ we suppose $P_{k}=\sum\limits_{i=0}^{k}\frac{M_{i}}{\lambda^{i}}.$ By Lemma $\ref{bmkm}$, noting that $N_{0}=0$, $M_{0}=\|u-v\|_{L^{\infty}(\Omega_{1})}$, then for any $k\geq0,$ we have
\begin{equation}\label{nn}
N_{k+1}\leq N_{k}+\widetilde{C}\frac{M_{k}}{\lambda^{k}}\leq\widetilde{C}P_{k},
\end{equation}
\begin{eqnarray*}
M_{k+1}\leq\lambda^{k}\sigma(\lambda^{k})+N_{k}\lambda^{k}\omega(\lambda^{k})+C_{1}(\lambda^{2}+\omega(\lambda^{k}))M_{k}+C_{2}\lambda^{k}\omega_{1}(\|u\|_{L^{\infty}(\Omega_{\lambda^{k}})}),
\end{eqnarray*}
where Definition $\ref{boundarydef}$, $\ref{bvdini}$ and Assumption $\ref{as1}$ are used.
Then
\begin{equation}\label{m2}
\frac{M_{k+1}}{\lambda^{k+1}}\leq\frac{1}{\lambda}\left(\sigma(\lambda^{k})+N_{k}\omega(\lambda^{k})\right)+\frac{C_{1}(\lambda^{2}+\omega(\lambda^{k}))}{\lambda}\left(\frac{M_{k}}{\lambda^{k}}\right)+\frac{C_{2}}{\lambda}\omega_{1}(\|u\|_{L^{\infty}(\Omega_{\lambda^{k}})}).
\end{equation}
Recalling the property of the modulus of continuity (see $(\ref{diniproperty})$) we have
\begin{eqnarray*}
\omega_{1}(\|u\|_{L^{\infty}(\Omega_{\lambda^{k}})})&\leq&\omega_{1}\left(\|u-v-N_{k}x_{n}\|_{L^{\infty}(\Omega_{\lambda^{k}})}+
\|v\|_{L^{\infty}(\Omega_{\lambda^{k}})}+\|N_{k}x_{n}\|_{L^{\infty}(\Omega_{\lambda^{k}})}\right)\\
&\leq&\omega_{1}(M_{k}+T\lambda^{k}+N_{k}\lambda^{k})\\
&\leq&2(\frac{M_{k}}{\lambda^{k}}+T+N_{k})\omega_{1}(\lambda^{k}).
\end{eqnarray*}
By substituting the above inequality and $(\ref{nn})$ into $(\ref{m2})$, we obtain
\begin{eqnarray*}
\frac{M_{k+1}}{\lambda^{k+1}}&\leq&\frac{1}{\lambda}\sigma(\lambda^{k})
+\frac{\widetilde{C}}{\lambda}\omega(\lambda^{k})P_{k-1}+
\frac{C_{1}(\lambda^{2}+\omega(\lambda^{k}))}{\lambda}\left(\frac{M_{k}}{\lambda^{k}}\right)
+\frac{2C_{2}}{\lambda}(N_{k}+\frac{M_{k}}{\lambda^{k}}+T)\omega_{1}(\lambda^{k})\\
&\leq&\frac{1}{\lambda}\sigma(\lambda^{k})+\frac{\widetilde{C}
+C_{1}}{\lambda}\omega(\lambda^{k})P_{k}+C_{1}\lambda\left(\frac{M_{k}}{\lambda^{k}}\right)
+\frac{2C_{2}\widetilde{C}}{\lambda}\omega_{1}(\lambda^{k})(P_{k}+T).\\
\end{eqnarray*}

We take $\lambda$ small enough firstly to make $C_{1}\lambda\leq\frac{1}{4}$, then we take $k_{0}$ large enough (then fixed) such that
$$
\sum_{i=k_{0}}^{\infty} \frac{\widetilde{C}+C_{1}}{\lambda}\omega(\lambda^{i}) \leq \frac{\widetilde{C}+C_{1}}{\lambda\ln \frac{1}{\lambda}}  \int_{0}^{\lambda^{k_{0}-1}} \frac{\omega(r)}{r} d r \leq \frac{1}{4},
$$
$$
\sum_{i=k_{0}}^{\infty} \frac{2C_{2}\widetilde{C}}{\lambda}\omega_{1}(\lambda^{i})\leq \frac{2C_{2}\widetilde{C}}{\lambda\ln \frac{1}{\lambda}}  \int_{0}^{\lambda^{k_{0}-1}} \frac{\omega_{1}(r)}{r} d r \leq \frac{1}{4}.
$$
For such $k_{0}(\geq 1),$ we have
$$
\sum_{i=k_{0}}^{\infty}\sigma(\lambda^{i}) \leq \frac{1}{\ln \frac{1}{\lambda}} \int_{0}^{1} \frac{\sigma(r)}{r} d r \leq \frac{1}{\ln \frac{1}{\lambda}}.
$$
Therefore for each $k \geq k_{0},$ we have
$$
\begin{aligned}
P_{k+1}-P_{k_{0}} &=\sum_{i=k_{0}}^{k} \frac{M_{i+1}}{\lambda^{i+1}} \\
&\leq\sum_{i=k_{0}}^{k}\frac{1}{\lambda}\sigma(\lambda^{i})
+\sum_{i=k_{0}}^{k}\frac{\widetilde{C}+C_{1}}{\lambda}\omega(\lambda^{i})P_{i}+\sum_{i=k_{0}}^{k}C_{1}\lambda\left(\frac{M_{i}}{\lambda^{i}}\right)
+\sum_{i=k_{0}}^{k}\frac{2C_{2}\widetilde{C}}{\lambda}\omega_{1}(\lambda^{i})(P_{i}+T)\\
& \leq \frac{1}{\lambda\ln \frac{1}{\lambda}}+P_{k+1}\left(\frac{1}{4}+\sum_{i=k_{0}}^{k}\frac{\widetilde{C}+C_{1}}{\lambda}\omega(\lambda^{i})\right)+(P_{k+1}+T)\sum_{i=k_{0}}^{k}\frac{2C_{2}\widetilde{C}}{\lambda}\omega_{1}(\lambda^{i}) \\
& \leq \frac{1}{\lambda\ln \frac{1}{\lambda}}+\frac{3}{4} P_{k+1}+\frac{1}{4}T.
\end{aligned}
$$
Then for all $k \geq k_{0}$,
$$
P_{k+1} \leq \frac{4}{\lambda\ln \frac{1}{\lambda}}+T+4 P_{k_{0}}.
$$
Therefore $\left\{P_{k}\right\}_{k=0}^{\infty}$ is bounded. We already proved $\sum\limits_{i=0}^{\infty} \frac{M_{i}}{\lambda^{i}}$ is convergent and $\left\{N_{i}\right\}_{i=0}^{\infty}$ is bounded.

Furthermore, by $(\ref{nn})$ and the definition of $P_{i}$ it's easy to see
$$
N_{i+1}-N_{i} \leq \widetilde{C} \frac{M_{i}}{\lambda^{i}}=\widetilde{C} P_{i}-\widetilde{C} P_{i-1}, \quad \text { for }~~ i \geq 1,
$$
and
$$
N_{i+1}-\widetilde{C}P_{i} \leq N_{i}-\widetilde{C} P_{i-1}, \quad \text { for }~~ i \geq 1.
$$
So $\left\{N_{i}-\widetilde{C}P_{i-1}\right\}_{i=1}^{\infty}$ is a bounded and non-increasing sequence and $\lim\limits_{i \rightarrow+\infty}(N_{i}-\widetilde{C}P_{i-1})$ exists. In conclusion $\lim\limits_{i \rightarrow+\infty} N_{i}$ exists and we set $\tau:=$
$\lim\limits_{i \rightarrow+\infty} N_{i} .$ The proof is finished.
\
\end{proof}
\begin{lm}\label{lim}$\lim\limits_{i\rightarrow+\infty}\frac{M_{i}}{\lambda^{i}}=0.$
\end{lm}
\begin{proof}
The proof is straightforward from Lemma \ref{con} since $\sum\limits_{i=0}^{\infty} \frac{M_{i}}{\lambda^{i}}$ is convergent.
\end{proof}
\begin{lm}For each $i=0,1,2, \ldots,~$ there exists $B_{i}$ such that $\lim\limits_{i \rightarrow \infty} B_{i}=0$ and that $$\left\|u-v-\tau x_{n}\right\|_{L^{\infty}(\Omega_{\lambda^{i}})} \leq B_{i} \lambda^{i}.$$
\end{lm}
\begin{proof}For any $i \geq 0$ we have
$$
\left\|u-v-\tau x_{n}\right\|_{L^{\infty}(\Omega_{\lambda^{i}})} \leq\|u-v-N_{i}x_{n}\|_{L^{\infty}(\Omega_{\lambda^{i}})}+\|N_{i}x_{n}-\tau x_{n}\|_{L^{\infty}(\Omega_{\lambda^{i}})}.
$$
Using $(\ref{induction})$ we get
$$
\left\|u-v-\tau x_{n}\right\|_{L^{\infty}(\Omega_{\lambda^{i}})}  \leq M_{i}+\lambda^{i}|N_{i}-\tau|.
$$
We set $B_{i}=\frac{M_{i}}{\lambda^{i}}+|N_{i}-\tau|$, then
$$
\left\|u-v-\tau x_{n}\right\|_{L^{\infty}(\Omega_{\lambda^{i}})}  \leq B_{i}\lambda^{i}.
$$
At the same time, by Lemma \ref{con} snd \ref{lim} we have
$$
\lim _{i \rightarrow \infty} B_{i}=0.
$$
The proof is completed.
\end{proof}
\noindent {\bf Proof of Theorem  \ref{mr1}}
From above four lemmas we already show that $u-v$ is differentiable at 0. Since $v$ is a Lipschitz function, it's clear that $u$ is Lipachitz at 0.
\section{Boundary Lipschitz regularity under Reifenberg $C^{1,\text{Dini}}$ condition}
In this section, we will generalize the results in Section 3 to Reifenberg $C^{1,\text{Dini}}$ domain. The main difficulty is the unit normal vectors are changing in different scales. But by Lemma $\ref{rem1.4}$, we notice that the difference of the unit vectors are controlled by Dini modulus of continuity and the unit vectors are convergent. The proof of Theorem $\ref{mr2}$ is similar to Theorem $\ref{mr1}$. We also use the following four lemmas to prove Theorem $\ref{mr2}$. For convenience, we only prove the boundary Lipschitz regularity at $0$. Without loss of generality, we can take $x_{0}=0$, denote $v_{0}$ by $v$ and assume that
$$\quad u(0)=g(0)=0,\quad v(0)=0,\quad \vec{a}=0, \quad r_{0}=1,
$$
$$\omega(1)\leq\lambda, \quad \int_{0}^{1} \frac{\omega(r)}{r} d r \leq 1, \quad \int_{0}^{1} \frac{\sigma(r)}{r} d r \leq 1, \quad \int_{0}^{1} \frac{\omega_{1}(r)}{r} d r \leq 1,
$$
where $\lambda$ will be determined in Lemma $\ref{kl}$ and Lemma $\ref{con2}$. In the following, we denote $\vec{n}_{\lambda^{i}}$ by $\vec{n}_{i}$.

\begin{lm}\label{bmkm2}
There exist nonnegative sequences $\{M_{i}\}_{i=0}^{\infty},$ and $\{N_{i}\}_{i=0}^{\infty},$ with $N_{0}=0$, $M_{0}=\|u-v\|_{L^{\infty}(\Omega_{1})}$, and for $i=0,1,2,\ldots,$
\begin{eqnarray*}
M_{i+1}&=&\|g-v-N_{i}x\cdot\vec{n}_{i}\|_{L^{\infty}(B_{\lambda^{i}}\cap\partial\Omega)}\\
&&+\|N_{i+1}x\cdot(\vec{n}_{i+1}-\vec{n}_{i})\|_{L^{\infty}(\Omega_{\lambda^{i}})}\\
&&+C_{1}(\lambda^{2}+\omega(\lambda^{i}))\|u-v-N_{i}x\cdot\vec{n}_{i}\|_{L^{\infty}(\Omega_{\lambda^{i}})}\\
&&+C_{2}\lambda^{i}\|\overrightarrow{\mathbf{F}}(x,u)-\overrightarrow{\mathbf{F}}(x,0)\|_{L^{\infty}(\Omega_{\lambda^{i}})},
\end{eqnarray*}
\begin{eqnarray*}
\label{n}|N_{i+1}-N_{i}|\leq\frac{\widetilde{C}}{\lambda^{i}}\|u-v-N_{i}x\cdot\vec{n}_{i}\|_{L^{\infty}(\Omega_{\lambda^{i}})},
\end{eqnarray*}
such that
\begin{equation}\label{induction2}
\|u-v-N_{i}x\cdot \vec{n}_i\|_{L^{\infty}(\Omega_{\lambda^{i}})}\leq M_{i}.
\end{equation}
\end{lm}
\begin{proof}
We prove this lemma inductively by using Lemma $\ref{kl}$, Remark $\ref{rem3.1}$ and $\ref{rem2.7}$ repeatedly.

When $i=0,$ since $N_{0}=0$ and $M_{0}=\|u-v\|_{L^{\infty}(\Omega_{1})}$,  it's easy to see
$$\|u-v-N_{0}x\cdot\vec{n}_{0}\|_{L^{\infty}(\Omega_{1})}= M_{0}.$$

When $i=1,$ by Definition $\ref{exrei}$, we have $B_{1}\cap\partial\Omega\subset\left\{|x\cdot\vec{n}_{0}|\leq \omega(1)\right\}$. Therefore by lemma $\ref{kl}$, there exists $0<N_{1}<\widetilde{C}\|u-v\|_{L^{\infty}(B_{1}\cap\Omega)}$ such that
\begin{eqnarray*}
\|u-v-N_{1}x\cdot\vec{n}_{0}\|_{L^{\infty}(B_{\lambda}\cap\Omega)}&\leq&
\|g-v\|_{L^{\infty}(B_{1}\cap\partial\Omega)}+C_{1}(\lambda^{2}+\omega(1))\|u-v\|_{L^{\infty}(B_{1}\cap\Omega)}\\
&&+C_{2}\|\overrightarrow{\mathbf{F}}(x,u)-\overrightarrow{\mathbf{F}}(x,0)\|_{L^{\infty}(B_{1}\cap\Omega)}.
\end{eqnarray*}
Then we have
\begin{eqnarray*}
\|u-v-N_{1}x\cdot\vec{n}_{1}\|_{L^{\infty}(B_{\lambda}\cap\Omega)}&\leq&
\|g-v\|_{L^{\infty}(B_{1}\cap\partial\Omega)}+\|N_{1}x\cdot(\vec{n}_{1}-\vec{n}_{0})\|_{L^{\infty}(B_{1}\cap\Omega)}\\
&&+C_{1}(\lambda^{2}+\omega(1))\|u-v\|_{L^{\infty}(B_{1}\cap\Omega)}+C_{2}\|\overrightarrow{\mathbf{F}}(x,u)-\overrightarrow{\mathbf{F}}(x,0)\|_{L^{\infty}(B_{1}\cap\Omega)}\triangleq M_{1}.
\end{eqnarray*}
and
$$|N_{1}-N_{0}|<\widetilde{C}\|u-v\|_{L^{\infty}(B_{1}\cap\Omega)}.$$
Next we assume that the conclusion is true for $i$. We consider the equation
$$
\left\{
\begin{array}{rcll}
\Delta (u-v-N_{i}x\cdot\vec{n}_{i})&=&\text{div}(\overrightarrow{\mathbf{F}}(x,u)-\overrightarrow{\mathbf{F}}(x,0))\qquad&\text{in}~~\Omega_{\lambda^{i}},\\
u-v-N_{i}x\cdot\vec{n}_{i}&=&g-v-N_{i}x\cdot\vec{n}_{i}\qquad&\text{on}~~B_{\lambda^{i}}\cap\partial\Omega. \\
\end{array}
\right.
$$
For $z=(z_{1}, z_{2}, \cdots,z_{n})\in\mathbb{R}^{n}$ we set
$$\widetilde{u}(z)=\frac{u(\lambda^{i} z)-v(\lambda^{i} z)-N_{i}\lambda^{i} z\cdot\vec{n}_{i}}{\lambda^{i}},
$$
$$\widetilde{g}(z)=\frac{g(\lambda^{i} z)-v(\lambda^{i} z)-N_{i}\lambda^{i} z\cdot\vec{n}_{i}}{\lambda^{i}},
$$
$$\widetilde{f}(z)=\overrightarrow{\mathbf{F}}(\lambda^{i} z,u(\lambda^{i} z))-\overrightarrow{\mathbf{F}}(\lambda^{i} z,0).
$$
Then $\widetilde{u}(z)$ is a solution of
$$
\left\{
\begin{array}{rcll}
\Delta \widetilde{u}(z)&=&\text{div}{\widetilde{f}}(z)\qquad&\text{in}~~B_{1}\cap\widetilde{\Omega},\\
\widetilde{u}(z)&=&\widetilde{g}(z)\qquad&\text{on}~~B_{1}\cap\partial\widetilde{\Omega}, \\
\end{array}
\right.
$$
where $\widetilde{\Omega}=\{z:\lambda^{i}z\in\Omega\}.$ Therefore $B_{1}\cap\partial\widetilde{\Omega}\subset B_{1}\cap\left\{|z\cdot\vec{n}_{i}|\leq \omega(\lambda^{i})\right\}$. Then by Remark $\ref{rem3.1}$, there exists a constant $K$ such that
\begin{eqnarray*}
\|\widetilde{u}-Kz\cdot\vec{n}_{i}\|_{L^{\infty}(B_{\lambda}\cap\widetilde{\Omega})}
&\leq&\|\widetilde{g}\|_{L^{\infty}(B_{1}\cap\partial\widetilde{\Omega})}+C_{1}(\lambda^{2}+\omega(\lambda^{i}))\|\widetilde{u}\|_{L^{\infty}(B_{1}\cap\widetilde{\Omega})}\\
&&+C_{2}\|{\widetilde{f}}\|_{L^{\infty}(B_{1}\cap\widetilde{\Omega})},
\end{eqnarray*}
where $K\leq \widetilde{C}\|\widetilde{u}\|_{L^{\infty}(B_{1}\cap\widetilde{\Omega})}=\frac{\widetilde{C}}{\lambda^{i}}\|u-v-N_{i}x\cdot\vec{n}_{i}\|_{L^{\infty}(\Omega_{\lambda^{i}})}.$
Let $N_{i+1}=N_{i}+K,$ scaling back, then we get
\begin{eqnarray*}
\|u-v-N_{i+1}x\cdot\vec{n}_{i+1}\|_{L^{\infty}(\Omega_{\lambda^{i+1}})}
&\leq&\|g-v-N_{i}x\cdot\vec{n}_{i}\|_{L^{\infty}(B_{\lambda^{i}}\cap\partial\Omega)}\\
&&+\|N_{i+1}x\cdot(\vec{n}_{i+1}-\vec{n}_{i})\|_{L^{\infty}(\Omega_{\lambda^{i}})}\\
&&+C_{1}(\lambda^{2}+\omega(\lambda^{i}))\|u-v-N_{i}x\cdot\vec{n}_{i}\|_{L^{\infty}(\Omega_{\lambda^{i}})}\\
&&+C_{2}\lambda^{i}\|\overrightarrow{\mathbf{F}}(x,u)-\overrightarrow{\mathbf{F}}(x,0)\|_{L^{\infty}(\Omega_{\lambda^{i}})}\triangleq M_{i+1},
\end{eqnarray*}
and
$$|N_{i+1}-N_{i}|=K\leq\frac{\widetilde{C}}{\lambda^{i}}\|u-v-N_{i}x\cdot\vec{n}_{i}\|_{L^{\infty}(\Omega_{\lambda^{i}})}.$$
This completes the proof of Lemma $\ref{bmkm2}$.
\end{proof}

\begin{lm}\label{con2}
$\sum\limits_{i=0}^{\infty} \frac{M_{i}}{\lambda^{i}}<\infty$ and $\lim \limits_{i \rightarrow \infty} N_{i}$ exists. We set
$$
\lim _{i \rightarrow \infty} N_{i}=\tau.
$$
\end{lm}
\begin{proof}
We assume $T$ is the Lipschitz constant respect to $v$, then
$$\|v\|_{L^{\infty}(B_{\lambda^{i}}\cap\partial\Omega)}=\|v-v(0)\|_{L^{\infty}(B_{\lambda^{i}}\cap\partial\Omega)}\leq T\lambda^{i}.
$$
For $k\geq0,$ we suppose $P_{k}=\sum\limits_{i=0}^{k}\frac{M_{i}}{\lambda^{i}}.$ By Lemma $\ref{bmkm2}$, noting that $N_{0}=0$ and $M_{0}=\|u-v\|_{L^{\infty}(\Omega_{1})}$, then for any $k\geq0$ we have
\begin{equation}\label{n2}
N_{k+1}\leq N_{k}+\widetilde{C}\frac{M_{k}}{\lambda^{k}}\leq\widetilde{C}P_{k},
\end{equation}
\begin{eqnarray*}
M_{k+1}\leq\lambda^{k}\sigma(\lambda^{k})+N_{k}\lambda^{k}\omega(\lambda^{k})+N_{k+1}\lambda^{k}S(\lambda)\omega(\lambda^{k})+C_{1}(\lambda^{2}+\omega(\lambda^{k}))M_{k}+C_{2}\lambda^{k}\omega_{1}(\|u\|_{L^{\infty}(\Omega_{\lambda^{k}})}),
\end{eqnarray*}
where Definition $\ref{exrei}$, $\ref{bvdini}$, Lemma $\ref{rem1.4}$ and Assumption $\ref{as1}$ are used.
Then
\begin{equation}\label{m22}
\frac{M_{k+1}}{\lambda^{k+1}}\leq\frac{1}{\lambda}\left(\sigma(\lambda^{k})+N_{k}\omega(\lambda^{k})+N_{k+1}S(\lambda)\omega(\lambda^{k})\right)+\frac{C_{1}(\lambda^{2}+\omega(\lambda^{k}))}{\lambda}\left(\frac{M_{k}}{\lambda^{k}}\right)+\frac{C_{2}}{\lambda}\omega_{1}(\|u\|_{L^{\infty}(\Omega_{\lambda^{k}})}).
\end{equation}
Recalling the property of the modulus of continuity (see ($\ref{diniproperty}$)) we have
\begin{eqnarray*}
\omega_{1}(\|u\|_{L^{\infty}(\Omega_{\lambda^{k}})})&\leq&\omega_{1}\left(\|u-v-N_{k}x\cdot\vec{n}_{k}\|_{L^{\infty}(\Omega_{\lambda^{k}})}+
\|v\|_{L^{\infty}(\Omega_{\lambda^{k}})}+\|N_{k}x\cdot\vec{n}_{k}\|_{L^{\infty}(\Omega_{\lambda^{k}})}\right)\\
&\leq&\omega_{1}(M_{k}+T\lambda^{k}+N_{k}\lambda^{k})\\
&\leq&2(\frac{M_{k}}{\lambda^{k}}+T+N_{k})\omega_{1}(\lambda^{k}).
\end{eqnarray*}
By substituting the above inequality and $(\ref{n2})$ into $(\ref{m22})$, we obtain
\begin{eqnarray*}
\frac{M_{k+1}}{\lambda^{k+1}}&\leq&\frac{1}{\lambda}\sigma(\lambda^{k})
+\frac{\widetilde{C}}{\lambda}\omega(\lambda^{k})(P_{k-1}+S(\lambda)P_{k})+
\frac{C_{1}(\lambda^{2}+\omega(\lambda^{k}))}{\lambda}\left(\frac{M_{k}}{\lambda^{k}}\right)
+\frac{2C_{2}}{\lambda}(N_{k}+\frac{M_{k}}{\lambda^{k}}+T)\omega_{1}(\lambda^{k})\\
&\leq&\frac{1}{\lambda}\sigma(\lambda^{k})+\frac{\widetilde{C}
+C_{1}+\widetilde{C}S(\lambda)}{\lambda}\omega(\lambda^{k})P_{k}+C_{1}\lambda\left(\frac{M_{k}}{\lambda^{k}}\right)
+\frac{2C_{2}\widetilde{C}}{\lambda}\omega_{1}(\lambda^{k})(P_{k}+T).\\
\end{eqnarray*}
The remaining proof is the same as Lemma $\ref{con}$, then we get $\sum\limits_{i=0}^{\infty} \frac{M_{i}}{\lambda^{i}}<\infty$ and $\lim \limits_{i \rightarrow \infty} N_{i}$ exists.

\end{proof}

\begin{lm}\label{lim2}$\lim\limits_{i\rightarrow+\infty}\frac{M_{i}}{\lambda^{i}}=0.$
\end{lm}
\begin{lm}For each $i=0,1,2, \ldots,~$ there exists $B_{i}$ such that $\lim\limits_{i \rightarrow \infty} B_{i}=0$ and that $$\left\|u-v-\tau x\cdot\vec{n}_{*}\right\|_{L^{\infty}(\Omega_{\lambda^{i}})} \leq B_{i} \lambda^{i},$$ where $\vec{n}_{*}$
is the limit of $\{\vec{n}_{\lambda^{i}}\}_{i=0}^{\infty}$ in Lemma $\ref{rem1.4}$.
\end{lm}
\begin{proof}For any $i \geq 0$ we have
\begin{eqnarray*}
\left\|u-v-\tau x\cdot\vec{n}_{*}\right\|_{L^{\infty}(\Omega_{\lambda^{i}})} &\leq&\|u-v-N_{i}x\cdot\vec{n}_{i}\|_{L^{\infty}(\Omega_{\lambda^{i}})}+\|N_{i}x\cdot\vec{n}_{i}-N_{i} x\cdot\vec{n}_{*}\|_{L^{\infty}(\Omega_{\lambda^{i}})}\\
&&+\|N_{i}x\cdot\vec{n}_{*}-\tau x\cdot\vec{n}_{*}\|_{L^{\infty}(\Omega_{\lambda^{i}})}
\end{eqnarray*}
Using $(\ref{induction2})$ we get
$$
\left\|u-v-\tau x\cdot\vec{n}_{*}\right\|_{L^{\infty}(\Omega_{\lambda^{i}})}  \leq M_{i}+N_{i}\lambda^{i}|\vec{n}_{i}-\vec{n}_{*}|+\lambda^{i}|N_{i}-\tau|.
$$
We set $B_{i}=\frac{M_{i}}{\lambda^{i}}+N_{i}|\vec{n}_{i}-\vec{n}_{*}|+|N_{i}-\tau|$, then
$$
\left\|u-v-\tau x\cdot\vec{n}_{*}\right\|_{L^{\infty}(\Omega_{\lambda^{i}})}  \leq B_{i}\lambda^{i}.
$$
At the same time, by Lemma \ref{con2} snd \ref{lim2} and Lemma $\ref{rem1.4}$ we have
$$
\lim _{i \rightarrow \infty} B_{i}=0.
$$
The proof is completed.
\end{proof}
\noindent {\bf Proof of Theorem  \ref{mr2}}
From above four lemmas we already show that $u-v$ is differentiable at 0. Since $v$ is a Lipschitz function, it's clear that $u$ is Lipachitz at 0.

\end{document}